%% file: main.tex
\documentclass{article}
\usepackage{amsmath,amsfonts,amssymb,amsthm,hyperref,xcolor,cleveref}
\usepackage{graphicx} 
\usepackage[T1]{fontenc}
\usepackage{textcomp}
\usepackage[margin=1in]{geometry}
\usepackage{mathtools}
\usepackage{enumitem}
\usepackage{float}
\usepackage{listings}
\usepackage{tikz}
\usepackage[style=alphabetic,backend=biber]{biblatex}
\definecolor{codedarkgreen}{RGB}{51, 133, 4}
\definecolor{codemaroon}{RGB}{133, 5, 63}
\definecolor{codeteal}{RGB}{0, 145, 109}
\definecolor{codepurple}{RGB}{123, 35, 125}
\newcommand{\GRevLex}{\textcolor{codeteal}{\texttt{GRevLex}}}

\newcommand{\GCAlgebra}{\textcolor{codedarkgreen}{\texttt{GCAlgebra }}}
\DeclareMathOperator{\Gr}{Gr}
\lstdefinelanguage{Macaulay2}{
basicstyle=\normalsize\ttfamily,
  alsoletter=",
  classoffset=1,
  keywords={coefficients,toList,transpose,det,factor,netList,subsets,genericMatrix,needsPackage,presentation,generators,gens,selectInSubring,i1,i2,i3,i4,i5,i6,i7,i8,i9,i10,i11,i17,i18,i19,i20,i21,flatten,ideal},
  keywordstyle={\color{blue}},
classoffset=2,
breaklines=true,
morekeywords={"SubalgebraBases","Engine","Brackets"},
keywordstyle={\color{codemaroon}},
classoffset=3,
morekeywords={QQ,CacheTable,Matrix,Eliminate},
keywordstyle={\color{codedarkgreen}},
classoffset=4,
morekeywords={restart,false,true,Weights,Limit,Lex,MonomialOrder,CoefficientRing},
keywordstyle={\color{codeteal}},
classoffset=5,
morekeywords={list,for,in,from,to,of},
keywordstyle={\color{codepurple}},
xleftmargin=1em,
xrightmargin=1em,
columns=fullflexible,
keepspaces=true,
stepnumber=1,
numbers=none,
captionpos=b,
showspaces=false,
frame=none
}

\newcommand{\GC}{\mathcal{G}}

\newcommand{\PP}{\mathbb{P}}

\usepackage{xcolor}
\usepackage{hyperref}

\DeclareMathOperator{\Sh}{Sh}

\theoremstyle{definition}
\newtheorem{definition}{Definition}[section]
\newtheorem{example}[definition]{Example}
\newtheorem{remark}[definition]{Remark}

\theoremstyle{plain}
\newtheorem{theorem}[definition]{Theorem}

\newcommand{\bQ}{{\mathbb{Q}}}
\newcommand{\bP}{\mathbb P}
\usepackage{amssymb}

\title{Brackets and Projective Geometry in Macaulay2}

\author{Dalton Bidleman, Timothy Duff, Jack Kendrick, Michael Zeng}

\begin{document}

\maketitle

\abstract{We introduce the \texttt{Brackets} package for the computer algebra system Macaulay2, which provides convenient syntax for computations involving the classical invariants of the special linear group. We describe our implementation of bracket rings and Grassmann-Cayley algebras, and illustrate basic functionality such as the straightening algorithm on examples from projective and enumerative geometry.}

\section{Introduction}\label{sec:intro}

Classical projective geometry is a source of beautiful results, such as the well-known configuration theorems of Pascal and Desargues, which describe conditions under which points, lines, conics, and other entities in projective space satisfy specified incidence conditions.
For automatic proofs of such theorems and other effective computations, it is convenient to recast geometry in an algebraic language: specifically, the invariant theory of the special linear group.
Generators for the ring of polynomial $\operatorname{SL}_d$-invariants are determinantal polynomials referred to as \emph{brackets}.
The classical straightening algorithm provides a procedure that rewrites arbitrary invariants in terms of these generators.

Incidences between linear subspaces of projective space determine an algebra all their own---the Grassmann-Cayley algebra---which further facillitates automatic theorem proving.
Such proofs follow a typical pattern. First, two geometric conditions are formulated in terms of expressions of the Grassmann-Cayley algebra. Second, these expressions are converted to straightened bracket polynomials. 
Finally, the two conditions are equivalent if and only if the straightened bracket polynomials are identical. 
This basic paradigm has a wealth of applications in computer vision~\cite{faugeras1995geometry,agarwal2025flatland}, robotics~\cite{white1994grassmann,thomas2023new}, and other subjects where geometric configurations play a role.

This article provides a brief overview of the background, functionality, and usage for the first version of the \texttt{Brackets} package in Macaulay2~\cite{M2}.
This first version lays the groundwork for eventual improvements in efficiency, additional functionality (eg.~invariants of binary forms, Cayley factorization), and interfacing with related packages such as \texttt{InvariantRing}~\cite{ferraro2024invariantring} and \texttt{SubalgebraBases}~\cite{burr2024subalgebrabases}.

\section{Background}\label{sec:background}

Our notation follows the exposition in~\cite[Ch.~3]{sturmfels2008invariant}.

\subsection{The Bracket Ring}\label{sub:brackets}

\newcommand{\xx}{X}
\newcommand{\kk}{\textbf{k}}

Fix a ground field $\kk$ and integers $n\ge d \ge 1.$ Let $X = (x_{i j})$ be an $n\times d$ matrix of distinct variables in the polynomial ring $\kk[\xx] := \kk [x_{i j}]
$ over a fixed field $k.$
We think of each row of $X$ as representing a point in the projective space $\PP^{d-1}$ of dimension $(d-1)$ over $k,$ so that $X$ represents a configuration of $n$ points in this projective space.
Many interesting geometric properties of this point configuration can be expressed in terms of the maximal minors of $X,$ which are conveniently written in \emph{bracket notation.}

A bracket $\lambda$ is a formal expression $[\lambda_1 \, \lambda_2 \, \ldots \, \lambda_d]$ where $1\le \lambda_1 < \lambda_2 < \ldots < \lambda_d \le n$ is a size $d$ subset of $\{1,2,\dots, n\}$. It represents the $d\times d$ minor of $X$ with rows indexed by the entries of the bracket.

\begin{example}\label{ex:collinear}
Let $n=4, d=3.$
The $4\times 3$ matrix
\[
X =\left(\begin{smallmatrix}
     x_{1,1}&x_{1,2}&x_{1,3}\\
     x_{2,1}&x_{2,2}&x_{2,3}\\
     x_{3,1}&x_{3,2}&x_{3,3}\\
     x_{4,1}&x_{4,2}&x_{4,3}
     \end{smallmatrix}\right)
\]
represents a configuration of $4$ points $x_1, \dots, x_4$ in the projective plane $\PP^2$, where each row of $X$ corresponds to the projective coordinate $x_i = (x_{i,1}: x_{i,2}: x_{i,3})$.

There are $\binom{4}{3}=4$ brackets, namely $[1 2 3], [1 2 4], [1 3 4], [2 3 4].$ A bracket $[abc]$ vanishes if and only if the points $x_a, x_b$, and $x_c$ are collinear. Thus, the condition that any three of the four points are collinear is expressed by the bracket equation
\[
[1 2 3] [1 2 4] [1 3 4] [2 3 4] = 0.
\]
\end{example}
Let $\Lambda(n,d):=\{[\lambda_1\cdots \lambda_d]\mid 1\leq \lambda_1 <\lambda_2 < \cdots <\lambda_d \leq n\}$ denote the set of brackets for $n$ points in $\PP^{d-1}$. Elements of the free polynomial algebra $\kk[\Lambda(d, n)]$ are known as {\em bracket polynomials}. There is a ring homomorphism that expresses bracket polynomials in terms of the entries of  $X$:
\begin{align}
\psi_{n,d}: \kk[\Lambda(n,d)] &\to \kk[\xx]\\
[\lambda_1 \, \ldots \, \lambda_d] &\mapsto \det \begin{pmatrix}
x_{\lambda_1, 1} & \cdots & x_{\lambda_1, d} \\
\vdots & \ddots  & \vdots \\
x_{\lambda_d, 1} & \cdots & x_{\lambda_d, d}
\end{pmatrix}
\end{align}
Abusing notation, we sometimes identify the bracket $\lambda \in \kk[\Lambda(n,d)]$ with its image $\psi_{n,d}(\lambda).$
Following~\cite{sturmfels2008invariant}, we call the image of $\psi_{n,d}$ the \textit{Bracket ring}, denoted $\mathcal{B}_{n,d}$, and let $I_{n,d}$ denote the kernel of $\psi_{n,d}$.
The $I_{n,d}$ ideal is generated by the well-known Pl\"{u}cker relations, and $\mathcal{B}_{n,d} \cong \kk[\Lambda(n,d)] / I_{n,d}$ is the homogeneous coordinate ring of the Grassmannian $\Gr(d,n)$, the set of $d$-dimensional subspaces of $\kk^n$, when viewed as a projective variety under the Pl\"{u}cker embedding.

\begin{example}
Let $n=4, d=2.$ The matrix
\[
X=\left(\begin{smallmatrix}
     x_{1,1}&x_{1,2}\\
     x_{2,1}&x_{2,2}\\
     x_{3,1}&x_{3,2}\\
     x_{4,1}&x_{4,2}
     \end{smallmatrix}\right)
\]
 represents a configuration of $4$ points on the projective line $\mathbb{P}^1.$
 
There are $\binom{4}{2}=6$ brackets. Unlike in Example \ref{ex:collinear}, the brackets are no longer algebraically independent since they satisfy the quadratic Pl\"{u}cker relation of $\Gr(2,4)$:
\begin{equation}\label{eq:pluecker}
[1 2] [3 4] - [1 3] [2 4] + [1 4] [2 3] = 0.
\end{equation}
\end{example}

In order to compute in $\mathcal{B}_{n,d}$, every bracket polynomial must be expressed by a canonical representative modulo the ideal $I_{n, d}.$ The classical {\em straightening algorithm} rewrites a bracket polynomial in such a way, which turns out to be the
normal form with respect to a certain Gr\"obner basis of the ideal $I_{n,d}.$

The monomial order $\prec$ on $\kk[\Lambda(n,d)]$ used in these Gr\"obner basis computations is known as the {\em tableau order}. A monomial in brackets is a \emph{tableau}, and can be visualized as an array of integers
\[
T = \begin{pmatrix}
  \lambda_1^1 & \ldots & \lambda_1^d \\
  \vdots & & \vdots \\
  \lambda_k^1 & \ldots & \lambda_k^d
  \end{pmatrix}.
\]
A tableau is \textit{standard} if its columns are sorted. An expression in brackets is said to be \emph{straightened} if every tableau appearing in it is standard. We order the set of brackets $\Lambda(n,d)$ lexicographically 
\begin{equation}
\lambda \prec \mu \text{ if for the smallest } i \text{ such that } \lambda_i\neq \mu_i, \text{ we have } \lambda_i < \mu_i
\end{equation} 
This ordering of variables specifies \GRevLex\ monomial order $\prec$ on $\kk[\Lambda(n,d)]$, the \emph{tableaux order}. Note that the standard tableaux form a vector space basis for $\mathcal{B}_{n,d}$. 

The first fundamental theorem of invariant theory \cite[Theorem 3.2.1]{sturmfels2008invariant} states that, for $\kk=\mathbb{C}$, the bracket ring $\mathcal{B}_{n,d}$ is the ring of polynomial invariants for the action of the special linear group $\operatorname{SL} (\mathbb{C}^d)$ by right-multiplication of $X.$ This already suggests a connection between $\mathcal{B}_{n,d}$ and the geometry of projective point configurations. The Grasssmann-Cayley algebra of the next section is an important tool for investigating these connections.

\subsection{Grassmann-Cayley Algebra}\label{sec:gc-algebra}

Traditionally (see eg.~\cite[\S 3.3]{sturmfels2008invariant}), the term Grassmann-Cayley algebra refers to the usual exterior algebra of a vector space $V$, endowed with an extra operation that represents the meet of two subspaces.
Traditionally, the symbol $\wedge $ is reserved for this meet operation. 
The usual wedge product, representing the join of two subspaces, may be denoted either by $\vee $ or simply $\cdot $ in this context.
We point out that the exterior algebra can be given the additional structure of a Hopf algebra, in which the meet operation serves as a comultiplication and the antipode map sends a vector in $V$ to its additive inverse.

For the purpose of automatically proving geometric incidence theorems, it is too restrictive to work with fixed vectors in a vector space $V$.
One would instead like to use $n$ formal variables to represent an \emph{arbitrary} configuration of $n$ points in the projective space $\PP (V).$
We now explain how this can be done.

The exterior algebra on a $n$-dimensional vector space may be realized as polynomial ring in $n$ skew-commuting variables over a ring $R.$
Following~\cite{stillman-exterior}, we denote this ring by $R \langle e_1, \ldots , e_n \rangle $.
The usual product in this ring corresponds to the join operation, in agreement with the Grassmann-Cayley notation conventions.
In our setting, the coefficient ring $R=\mathcal{B}_{n,d}$ will be the bracket ring of the previous section. 
Gr\"{o}bner basis computations such as normal forms in the ring $ \langle e_1, \ldots , e_n \rangle $ by a simple adaptation of Buchberger's algorithm.
We define the \emph{Grassmann-Cayley ring} for configurations of $n$ points in $\PP^{d-1}$,
\begin{equation}\label{eq:gc-ring-new}
\GC_{d} (e_1, \ldots , e_n) = \mathcal{B}_{n,d} \langle e_1, \ldots , e_n \rangle / 
J_{n,d},
\end{equation}
where $J_{n,d}$ is the (two-sided) ideal generated by all squarefree monomials $e_{i_0} \cdots e_{i_d}$ of degree $d+1.$
In this setting, the join operation in $\GC_{d} (e_1, \ldots , e_n)$ is inherited naturally from polynomial multiplication in $\mathcal{B}_{n,d} \langle e_1,\ldots , e_n \rangle .$ 
Monomials in $\mathcal{B}_{n,d} \langle e_1,\ldots , e_n \rangle $ are known as \emph{blades}, and we refer to their residue classes in
$\GC_{d} (e_1, \ldots , e_n)$ as \emph{extensors}.
Thus,
$\GC (e_1,\ldots, e_n)$ is a noncommutative algebra over $\mathcal{B}_{n,d}$ whose nonzero graded pieces $\GC_d (e_1,\ldots, e_n)^{(k)}$, $0\le k \le d,$ are spanned by the degree-$k$ extensors.

Extensors of degree $d$ correspond to brackets in a natural way:
\begin{align*}
\GC_d (e_1,\ldots , e_n)^{(d)} \ni e_{i_1} \cdots e_{i_d} \leftrightarrow [i_1 \, \ldots \, i_d] \in \Lambda (n,d).
\end{align*}
To prevent accumulation of indices when defining the meet operation, we may also write the bracket corresponding to an extensor as $[e_{i_1} \, \ldots \, e_{i_d}].$
Extending by distributivity, it will be sufficient to define the meet on extensors.
For extensors $a=a_1 a_2 \cdots a_j$ and $b=b_1 b_2 \cdots b_k$ where $j+k \geq d$, we define their meet $a\wedge b$ as in~\cite{sturmfels2008invariant} using the ``shuffle product",
\begin{equation}\label{eq:shuffle-product}
a \wedge b=\sum_{w\in \Sh_{j,k,d} }\operatorname{sgn}(w)\left[a_{w_1}\,  \ldots \, a_{w_{d-k}}\,  b_1 \, \ldots \,  b_k\right] a_{w_{d-k+1}} \cdots a_{w_j}
\in \GC_d^{(j+k-d)} (e_1, \ldots , e_n),
\end{equation}
where the sum in~\eqref{eq:shuffle-product} is taken over the set of \emph{shuffle permutations} written in one-line notation,
\begin{equation}\label{eq:shuffle-set}
\Sh_{j,k,d} = 
\left\{ w = (w_1, \ldots , w_j) \in \mathcal{S}_j \mid w_1 < w_2 < \ldots < w_{d-k}, \, \, w_{d-k+1} < w_{d-k+2} < \ldots < w_j \right\} .
\end{equation}
Just as the join of two extensors corresponds to taking the (projective) span of two linear spaces, the meet $a \wedge b$ of two extensors is corresponds to taking the intersection of subspaces.
We refer to~\cite[Thm 3.2.2]{sturmfels2008invariant} for a justification of this and other facts in the setting of classical Grassmann-Cayley algebras, such as the following anti-commutativity relations: if $a$ and $b$ are extensors of ranks $j$ and $k$, then $a\cdot b = (-1)^{jk}b\cdot a$, and $a\wedge b = (-1)^{(d-j)(d-k)}b \wedge a$.
We point out one slight difference between the classical setup and ours: as shown in~\Cref{ex:first-gc} the shuffle product if two extensors is generally not an extensor according to the definition above.
Nevertheless, for any choice of vectors $v_1, \ldots , v_n $ in a $d$-dimensional vector space $V,$ evaluating the expression~\eqref{eq:shuffle-product} at $e_i=v_i$ yields the classical shuffle product.

We note that, in the particular case where $j+k=d,$ the shuffle product of degree $j$ and $k$ extensors in equation~\eqref{eq:shuffle-product} produces an element of the bracket ring $\mathcal{B}_{n,d}.$

\begin{example}\label{ex:first-bracket}
Points in projective space correspond to $1$-extensors. We formulate the condition that a point $e_1$ lies on a line $\overline{e_2 e_3}$ using the Grassman-Cayley algebra. 
The line $\overline{e_2 e_3}$ is represented by the extensor $e_2\cdot e_3$, and the condition that $e_1$ lies on $\overline{e_2 e_3}$ (i.e. $e_1$ meets $\overline{e_2 e_3}$) corresponds to the vanishing of the bracket $$e_1\wedge e_2 e_3 = [e_1 \, e_2 \, e_3 ] = 0.$$
\end{example}

\newcommand{\Brackets}{\texttt{Brackets} }

In~\Cref{sec:Examples}, we use the \Brackets package to further illustrate how incidence theorems in synthetic projective geometry can be derived utilizing the formalism of Grassmann-Cayley algebras.

\begin{remark}
The field $\kk$ may be replaced with a more general coefficient ring when defining the rings $\mathcal{B}_{n,d}$ and $\GC_d (e_1, \ldots , e_n)$ and the shuffle product.
In such cases, the option \texttt{CoefficientRing} can be supplied to the ring constructors described below. See~\Cref{sub:transversals} for a specific example.
\end{remark}


\section{Data Types and Basic Usage}\label{sec:functionality}

Formation of the bracket rings $\mathcal{B}_{n,d}$ and the Grassman-Cayley rings $\GC_d (e_1, \ldots , e_n)$ defined in~\eqref{eq:gc-ring-new}, as well as implementing the straightening algorithm on $\mathcal{B}_{n,d}$ and the shuffle product~\eqref{eq:shuffle-product} on $\GC_d (e_1, \ldots , e_n)$, are all straightforward tasks that can be accomplished with the core functionality of Macaulay2.
The \Brackets package provides dedicated datatypes and special syntax for manipulating GC expressions, with a view towards maintaining the notational elegance of the Grassmann-Cayley formalism.

Our package provides three main datatypes: 
\begin{enumerate}
    \item \texttt{BracketRing}, for representing the rings $\mathcal{B}_{n,d}$,
    \item \texttt{GCRing}, for representing the rings $\GC_d (e_1, \ldots , e_n)$, and
    \item \texttt{GCExpression}, representing elements of the ring $\GC_d (e_1, \ldots , e_n)$.
\end{enumerate}
These objects are implemented as hash tables pointing to instances of the standard \texttt{Ring} or \texttt{RingElement} types.
Both \texttt{BracketRing} and \texttt{GCRing} inherit from a parent class \texttt{AbstractGCRing}.
Objects of class \texttt{GCExpression} may be obtained by conversion from corresponding \texttt{RingElement} instances; once these expressions are created, new expressions can be formed via various operations.
To handle these conversions gracefully, we define a new method \texttt{(\_, RingElement, AbstractGCRing)} for Macaulay2's native subscript operator.
This subscript operator is essential for performing bracket and GC ring computations.

\begin{example}\label{ex:first-gc}
(\cite[Example 3.1.10]{sturmfels2008invariant})
Let us form the bracket $[1\, 4 \, 5] \in \mathcal{B}_{6,3}$:
\begin{lstlisting}[language=macaulay2]
i1 : needsPackage "Brackets";
i2 : B = bracketRing(6, 3)
o2 = B
      6,3
o2 : BracketRing
i3 : [1 4 5]_B
o3 = [145]
o3 : Bracket
\end{lstlisting}
The constructor \texttt{bracketRing} offers users the option of inputting their own symbols to replace the integers $1, \ldots , n$ appearing inside of brackets.
New GC expressions in $\mathcal{B}_{n,d}$ can be formed through addition, multiplication, and scalar multiplication as with usual polynomials.
The following example illustrates the straightening algorithm applied to a bracket monomial.
\begin{lstlisting}[language=macaulay2]
i4 : T = [1 4 5]_B * [1 5 6]_B * [2 3 4]_B
o4 = [234]*[156]*[145]
o4 : GCExpression
i5 : normalForm T 
o5 = [256]*[145]*[134]-[356]*[145]*[124]+[456]*[145]*[123]
o5 : GCExpression
\end{lstlisting}
\end{example}

\begin{example} (\cite[Ex 3.3.3]{sturmfels2008invariant})
To illustrate interactions between GC rings and their associated brackets, we study the GC expression representing the intersection of three lines $\overline{ad}, \overline{be}, \overline{cf} \in \PP^2$:
\begin{equation}\label{eq:gc-3lines}
ad \wedge be \wedge cf \in \GC_3 (a, \ldots , f).
\end{equation}
First, we form the underlying GC ring,
\begin{lstlisting}[language=macaulay2]
i2 : G = gc(a..f, 3)
o2 = Grassmann-Cayley Algebra generated by 1-extensors a..f
\end{lstlisting}
The variables $a, \ldots , f$ are interpreted by Macaulay2 to be traditional ring elements. 
To work with the corresponding GC expression, we must again utilize the subscript operator:
\begin{lstlisting}[language=Macaulay2]
i3 : instance(a, GCExpression)
o3 = false
i4 : instance(a_G, GCExpression)
o4 = true
\end{lstlisting}
To form more complex expressions, addition and multiplication of either GC expressions or their underlying ring elements can be used, as shown below.
\begin{lstlisting}[language=Macaulay2]
i5 : a_G * b_G
o5 = a*b
o5 : GCExpression
i6 : (a*b)_G
o6 = a*b
o6 : GCExpression
\end{lstlisting}
Although we do not provide an equality operator for \texttt{GCExpression}, we can check an important special case; two bracket polynomials are equal in $\mathcal{B}_{n,d}$ if and only if their difference straightens to zero.

The shuffle product $\wedge $ for a pair of GC expressions is implemented using the operator \texttt{\^}.
We illustrate first the shuffle product of two degree-$2$ extensors:
\begin{lstlisting}[language=Macaulay2]
i7 : A = (a * d)_G;
i8 : B = (b * e)_G;
i9 : AB = A ^ B 
o9 = [bde]*a+[abe]*d
o9 : GCExpression
\end{lstlisting}
In the classical Grassman-Cayley algebra, the GC expression $ad \wedge be$ would be a degree-$1$ extensor, representing the unique point in the intersection $\overline{ad} \cap \overline{be} \subset \PP^2$ as long as $\overline{ad} \ne \overline{be}.$
Here, we obtain a $\mathcal{B}_{6,3}$-linear combination of the 1-extensors spanning $\overline{ad}$, which agrees with the classical result upon specialization.
The weights in this combination may be determined using Cramer's rule.

Finally, we obtain a bracket formula for the GC expression~\eqref{eq:gc-3lines}:
\begin{lstlisting}[language=Macaulay2]
i9 : C = (c * f)_G
o6 = c*f
o6 : GCExpression
i7 : D = AB ^ C
o7 = 2*[bde]*[acf]-2*[cdf]*[abe]
o7 : GCExpression
\end{lstlisting}
This bracket polynomial vanishes if and only if the three lines intersect, ie.~$\overline{ad} \cap \overline{be} \cap \overline{cf} \ne \emptyset .$
\end{example}

In addition to the various ring-like operations involving GC expressions, the \Brackets package also allows users to work with polynomials in the entries of a $n\times d$ matrix $X,$ which may be (partially) rewritten in terms of bracket polynomials so as to test membership in $\mathcal{B}_{n,d}$, via methods like \texttt{toBracketPolynomial}, and \texttt{normalForm} for the straightening algorithm,
We refer the following examples and the package's documentation for sample usage.

\section{Further Examples}\label{sec:Examples}
\subsection{Desargues' Theorem}\label{sub:Desargues}

In this section, we use \Brackets\ to derive Desargues' classical theorem on perspective triangles. 

Fix six distinct points $a, b, c, d, e, f$ in $\PP^2$ and consider the two triangles $\triangle abc$ and $\triangle def.$ The two triangles are said to be perspective from a point if the three straight lines connecting corresponding vertices in each triangle all intersect at a single point. On the other hand, the triangles are perspective from a line if the intersection points of each pair of corresponding sides are collinear.

\begin{figure}[h]
    \centering
\begin{tikzpicture}[x=0.8pt,y=0.8pt,yscale=-1,xscale=1, line width=0.8]

\draw[gray,dashed]   (140,160) -- (230,250) ; 
\draw[gray,dashed]   (140,160) -- (140,310) ; 
\draw[gray,dashed]   (140,160) -- (50,250) ; 

\draw[gray] (50,250) -- (290,250); 
\draw[gray] (50,250) -- (230,370) ; 
\draw[gray] (140,310) -- (320,190) ; 
\draw[gray,dashed] (330,170) -- (220,390) node[pos=0.5, right]{\color{black}\large$\ell$};
\draw[gray] (110,190) -- (290,250) ;
\draw[gray] (110,190) -- (230,370) ;
\draw[gray] (320,190) -- (140,235) ;

\filldraw[black] (290,250) circle (2pt); 
\filldraw[black] (230,370) circle (2pt); 
\filldraw[black] (320,190) circle (2pt); 

\filldraw[black] (230,250) circle (2pt); 
\filldraw[black] (140,310) circle (2pt); 
\filldraw[black] (50,250) circle (2pt);

\filldraw[black] (140, 235) circle (2pt); 
\filldraw[black] (110,190) circle (2pt); 
\filldraw[black] (200,220) circle (2pt); 

\filldraw[black] (140,160) circle (2pt) node[anchor=south]{$O$}; 

\draw[fill = {rgb, 255:red, 74; green, 144; blue, 226 }, thick, fill opacity=0.65, text opacity=1] (230, 250) node[anchor=north]{\large $d$}
-- (140, 310) node[anchor=north]{\large $e$}
-- (50, 250) node[anchor=east]{\large $f$}
-- cycle;

\draw[fill = {rgb, 255:red, 208; green, 2; blue, 27 }, thick, fill opacity=0.65, text opacity=1] (110, 190) node[anchor=east]{\large $c$}
-- (140, 235) node[anchor=east]{\large $b$}
-- (200, 220) node[anchor=south]{\large $a$}
-- cycle;

\end{tikzpicture}

    \caption{Two perspective triangles $\triangle abc$ and $\triangle def$. The lines connecting corresponding vertices of each triangle all intersect at the point $O$, whereas pairs of corresponding sides all intersect at points on the line $\ell.$}
    \label{fig:desargues}
\end{figure}
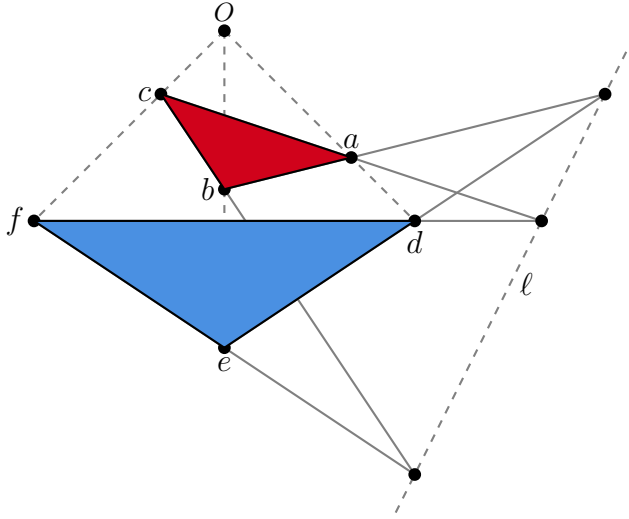

Figure~\ref{fig:desargues} depicts a pair of triangles that are perspective from both a point and a line. While the two different notions of perspectivity have different definitions, Desargues' theorem states that these two conditions are in fact equivalent.

\begin{theorem}[Desargues]
    Two triangles $\triangle abc$ and $\triangle def$ of points in $\PP^2$ are perspective from a point if and only if they are perspective from a line.
\end{theorem}

\begin{proof}
We provide a simple proof of this theorem using the basic functionality of \Brackets. First, we initialize the Grassmann-Cayley algebra for the six points $a,b,c,d,e,f$ in $\PP^2.$

\begin{lstlisting}[language=Macaulay2]
needsPackage "Brackets"
G = gc(a..f,3) -- Grassmann-Cayley algebra for 6 points in P^2
\end{lstlisting}
The lines spanned by each pair of vertices form the sides of each triangle. Note that the command \texttt{\_G} ensures that each line is considered as an element of the Grassmann-Cayley algebra.
\begin{lstlisting}[language=Macaulay2]
abLine = (a * b)_G -- line spanned by a and b
bcLine = (b * c)_G -- line spanned by b and c
acLine = (a * c)_G -- line spanned by a and c
deLine = (d * e)_G -- line spanned by d and e
efLine = (e * f)_G -- line spanned by e and f
dfLine = (d * f)_G -- line spanned by d and f
\end{lstlisting}

Without loss of generality, we set that vertex $a$ corresponds to vertex $d,$ vertex $b$ to vertex $e,$ and vertex $c$ to vertex $f.$ We now form the condition that the two triangles in perspective from a line. For a given pair of corresponding sides, the intersection point is generated by the meet operation, $\wedge.$ The expression \texttt{linePerspective} is the join of the three intersection points and is equal to zero if and only if the points are collinear. Thus, the two triangles $\triangle abc$ and $\triangle def$ are perspective from a line if and only if \texttt{linePerspective} vanishes.
\begin{lstlisting}[language=Macaulay2]
pt1 = abLine ^ deLine -- intersection of ab and de
pt2 = bcLine ^ efLine -- intersection of bc and ef
pt3 = acLine ^ dfLine -- intersection of ac and df
linePerspective = pt1 * pt2 * pt3 -- Condition that the pts p1, p2, p3 are collinear
\end{lstlisting}

Next, we form the condition that the triangles are perspective from a point. We first form the lines spanned by corresponding pairs of vertices. The expression \texttt{pointPerspective} is the meet of these three lines and vanishes if and only if the three lines intersect at a single point. It follows that the triangles $\triangle abc$ and $\triangle def$ are perspective from a point if and only if \texttt{pointPerspective} vanishes.
\begin{lstlisting}[language=Macaulay2]
adLine = (a * d)_G -- line spanned by a and d
beLine = (b * e)_G -- line spanned by b and e
cfLine = (c * f)_G -- line spanned by c and f
pointPerspective =  adLine ^ beLine ^ cfLine
\end{lstlisting}
The expressions \texttt{linePerspective} and \texttt{pointPersective} are two degree zero elements of the Grassmann-Cayley algebra, which we identify with elements of the bracket ring $\mathcal{B}_{6, 3}.$ The representatives of these elements in the bracket ring do not share any common factors:
\begin{lstlisting}[language=Macaulay2]
i17 : factor linePerspective 
o17 : {[abc], [cef]*[bde]*[adf]-[cdf]*[bef]*[ade]}
i18 : factor pointPerspective 
o18 : {[bde]*[acf]-[cdf]*[abe], 2}
\end{lstlisting}
However, applying the straightening algorithm to each expression via the \texttt{normalForm} command we can rewrite the expressions in normal form. 
\begin{lstlisting}[language=Macaulay2]
i19 : nl = normalForm linePerspective 
o19 : [def]*[bdf]*[ace]*[abc]-[def]*[bef]*[acd]*[abc]-[def]*[cdf]*[abe]*[abc]-[def]^2*[abc]^2
i20 : np = normalForm pointPerspective 
o20 : 2*[bdf]*[ace]-2*[bef]*[acd]-2*[cdf]*[abe]-2*[def]*[abc]
\end{lstlisting}
Thus, by applying the straightening algorithm to \texttt{pointPerspective} and \texttt{linePerspective} we have
\begin{lstlisting}[language=Macaulay2]
i21 : [abc] * [def] * nl - 2 * nl
o21 : 0
\end{lstlisting}
Thus, the two triangles are perspective from a point if and only if either one of the triangles $\triangle abc, \triangle def$ are degenerate or the triangles are perspective from a line, concluding our proof of Desargues' theorem.
\end{proof}

\subsection{Transversals of Four Lines in Space}\label{sub:transversals}

A \textit{transversal} to a number of given lines $\ell_1, \ell_2, \dots, \ell_n$ in $\PP^3$ is a line which intersects all $\ell_i$ non-trivially. Counting the transversals of four skew lines in $\PP^3$ is a famous problem in classical projective geometry. 
We use the \Brackets package to rederive this classical result. Through this example, we also demonstrate how to work with extra formal parameters in the coefficient ring.

\begin{theorem}
    There are two transversals to four general lines in $\bP^3$. 
\end{theorem}

Our strategy is to represent each of the four lines as the join of two points in the Grassmann-Cayley algebra and then express the transversal as an incidence relation. Fixing one of the lines as $$\ell_1 = a\cdot b$$ for $a\neq b\in \bP^3$, then the common transversal meets $\ell_1$ at some point $p\in \ell_1$, which has the form $$p = \lambda a + \mu b, \quad \lambda + \mu = 1.$$
This prompts the introduction of scalar parameters $\lambda$ and $\mu$ in the coefficient ring of the Grassmann-Cayley algebra. We initialize the Grassmann-Cayley algebra for eight general points in $\PP^3$:

 \begin{lstlisting}[language=Macaulay2]
needsPackage "Brackets"
G = gc(toList(a..h), 4, CoefficientRing => QQ[l,m])
 \end{lstlisting}
The argument $4 = 3+1$ sets the ambient space as $\bP^3$, and the coefficient ring is set to be $\bQ[\lambda, \mu]$, by slight abuse of notation. Next, we create each of the four lines as a join of two points:

\begin{lstlisting}[language = Macaulay2]
ell1 = (a*b)_G -- line spanned by a and b
ell2 = (c*d)_G -- line spanned by c and d
ell3 = (e*f)_G -- line spanned by e and f
ell4 = (g*h)_G -- line spanned by g and h
\end{lstlisting}

The command \textcolor{codedarkgreen}{\texttt{\_G}} ensures that the expressions are considered as honest elements of the \GCAlgebra.  A transversal $\ell$ will meet $\ell_1$ at some point $p = \lambda a + \mu b$
 as previously discussed. 
\begin{lstlisting}[language = Macaulay2]
p = (l*a + m * b)_G
\end{lstlisting}
Since the lines are general, there exist values $\lambda, \mu$ such that $p$ does not lie on $\ell_2$, so the join $V = p\cdot \ell_2$ is the blue plane in \Cref{fig:transversals}. Since the transversal $\ell$ also intersects $\ell_2$, the plane $V$ contains both $p$ and the intersection point $\ell\wedge \ell_2.$ It follows that $\ell$ is contained in $V$. 

Now, the plane $V$ meets the general line $\ell_3\subset \bP^3$ in exactly one point $q = V\wedge\ell_3$. Again, since the transversal $\ell$ intersects $\ell_3$ and $\ell$ is contained in $V,$ the point $q$ has to be the intersection point $\ell\wedge \ell_3.$ Thus, $q$ lies on $\ell$ and the transversal is the join of $q$ and $p$, $$q\cdot p = ((p\cdot \ell_2)\wedge \ell_3)\cdot p.$$
The \Brackets expression for the transversal $\ell$ is 
\begin{lstlisting}[language = Macaulay2]
ell = ((p * ell2) ^ ell3) * p
\end{lstlisting}

\begin{figure}
    \input{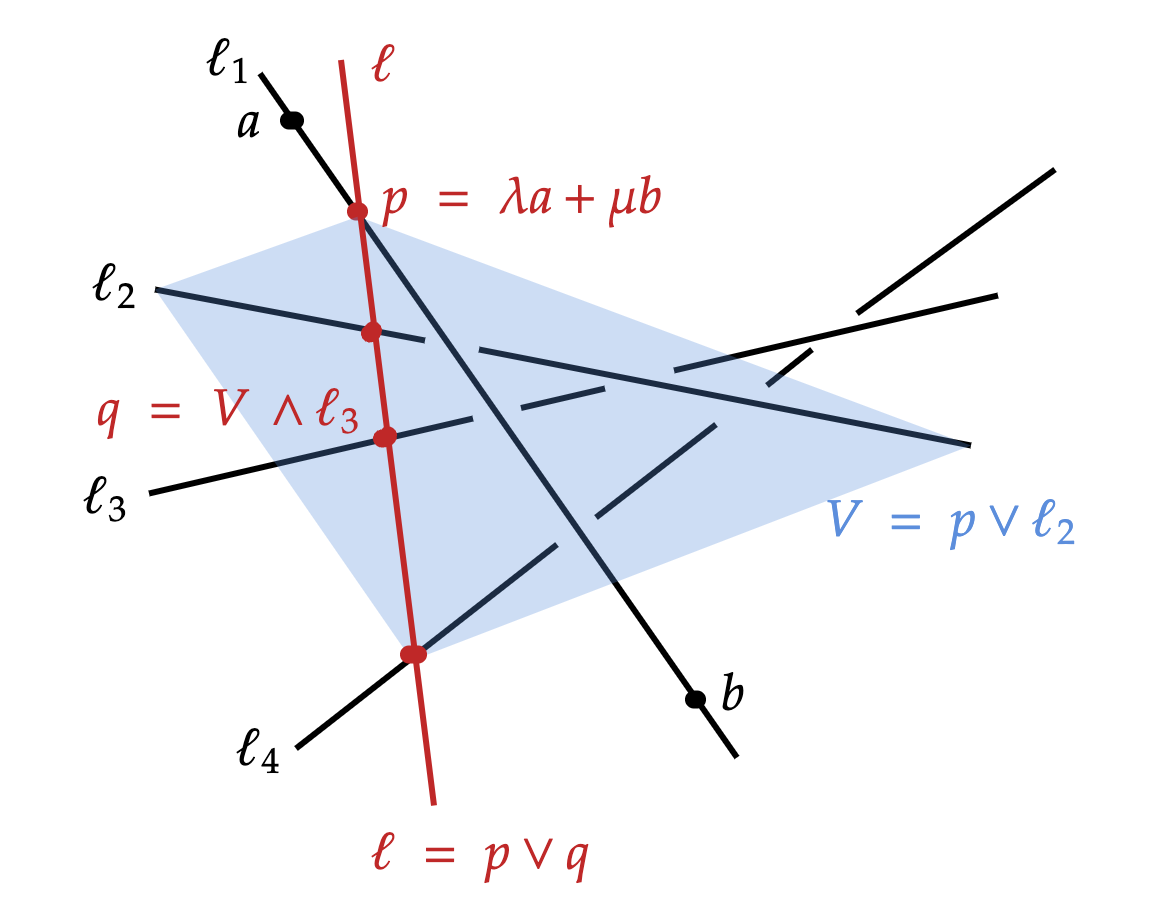}
    \centering 
    \caption{Construction of the transversal $\ell$ using the Grassmann-Cayley algebra.}\label{fig:transversals}
\end{figure}

The line $\ell$ intersects the remaining line $\ell_4$ if and only if the expression
\begin{lstlisting}[language = Macaulay2]
formula = ell * ell4
\end{lstlisting}
vanishes. We can compute it as the following \Brackets expression:
\begin{lstlisting}[language = Macaulay2]
i9 : formula

o9 = m^2*[bdef]*[bcgh]-2*m^2*[bdgh]*[bcef]-2*l*m*[bcef]*[adgh]+l*m*[bcgh]*[adef]+l*m*[bdef]*[acgh]+l^2*[adef]*[acgh]-2*l*m*[bdgh]*[acef]-2*l^2*[adgh]*[acef]

o9 : GCExpression
\end{lstlisting}
Note that we can view the above as a quadratic equation in the variables $\lambda,\mu$, with $\lambda+\mu = 1$. The number of transversals of the four lines $\ell_1, \ldots, \ell_4$ is exactly the number of roots of this quadratic. To determine the number of roots, we compute the discriminant and so must extract the coefficients of $\lambda^2, \lambda\mu, \mu^2$. This is done as follows: 

\begin{lstlisting}[language = Macaulay2]
i10 : (m, c) = coefficients formula;
i11 : disc = c_(2,0) * c_(0,0) - 4 * c_(1,0)

o11 = [bdef]*[bcgh]*[adef]*[acgh]-2*[bdgh]*[bcef]*[adef]*[acgh]-2*[bdef]*[bcgh]*[adgh]*[acef]+4*[bdgh]*[bcef]*[adgh]*[acef]+8*[bcef]*[adgh]-4*[bcgh]*[adef]-4*[bdef]*[acgh]+8*[bdgh]*[acef]

o11 : GCExpression
\end{lstlisting}
The discriminant is generically non-zero, which implies that there are two transversals to four general lines in $\bP^3$. If the discriminant is equal to zero, there is only one transversal. This occurs when the following bracket polynomial vanishies:

\begin{lstlisting}[language = Macaulay2]
[bdef]*[bcgh]*[adef]*[acgh]-2*[bdgh]*[bcef]*[adef]*[acgh]-2*[bdef]*[bcgh]*[adgh]*[acef]+4*[bdgh]*[bcef]*[adgh]*[acef]+8*[bcef]*[adgh]-4*[bcgh]*[adef]-4*[bdef]*[acgh]+8*[bdgh]*[acef]
\end{lstlisting}
Finally, if each coefficient in front of $\lambda$ and $\mu$ is zero, there are infinitely many transversals. This occurs when the following \emph{three} bracket polynomials vanish (namely, the entries of \texttt{c} above.)

\section*{Acknowledgments}

This project began during a reading group on invariant theory at the University of Washington in 2022. We thank the other participants in this group.
Work continued at the 2023 Macaulay2 workshop in Minneapolis, funded by NSF DMS 2302476.
We thank the workshop organizers and participants, particularly Thomas Yahl who provided helpful input on the development of this package. 
Duff acknowledges partial support from an NSF Mathematical Sciences Postdoctoral Fellowship (DMS 2103310.)

\printbibliography
\par\vspace{\baselineskip}\noindent\small
\begin{tabular}{@{}l}
\scshape\ Department of Mathematics and Statistics, Auburn University\\\textit{E-mail address:} \url{deb0036@auburn.edu}\\
\scshape\ Department of Mathematics, University of Missouri - Columbia\\\textit{E-mail address: }\url{tduff@missouri.edu}\\
\scshape\ Department of Mathematics, University of Washington\\\textit{E-mail address:} \url{jackgk@uw.edu} \\
\scshape\ Department of Mathematics, University of Washington\\\textit{E-mail address:} \url{zengrf@uw.edu}
\end{tabular}
\end{document}

%% file: transversals.tex
\tikzset{every picture/.style={line width=0.75pt}} 

\begin{tikzpicture}[x=0.75pt,y=0.75pt,yscale=-1,xscale=1]

\draw  [color={rgb, 255:red, 74; green, 144; blue, 226 }  ,draw opacity=1 ][fill={rgb, 255:red, 74; green, 144; blue, 226 }  ,fill opacity=0.3 ][dash pattern={on 0.84pt off 2.51pt}] (276.14,69.36) -- (479,145.83) -- (293.14,216.36) -- (207,93.83) -- cycle ;
\draw [line width=0.75]    (242,21.83) -- (401,249.83) ;
\draw [line width=0.75]    (207,93.83) -- (297,110.83) ;
\draw [line width=0.75]    (315,113.83) -- (479,145.83) ;
\draw [line width=0.75]    (205,161.83) -- (313,136.83) ;
\draw [line width=0.75]    (380,120.83) -- (488,95.83) ;
\draw [line width=0.75]    (329,133.33) -- (357,126.83) ;
\draw [line width=0.75]    (254,246.83) -- (296.67,213.48) -- (341,178.83) ;
\draw [line width=0.75]    (354,168.83) -- (394,137.83) ;
\draw [line width=0.75]    (411,125.83) -- (426,113.83) ;
\draw [line width=0.75]    (441,101.83) -- (507,53.83) ;
\draw [color={rgb, 255:red, 208; green, 2; blue, 27 }  ,draw opacity=1 ][line width=0.75]    (269,17.33) -- (300,265.83) ;
\draw  [line width=3.75] [line join = round][line cap = round] (252.14,37.36) .. controls (252.14,37.36) and (252.14,37.36) .. (252.14,37.36) ;
\draw  [line width=3.75] [line join = round][line cap = round] (383.14,225.36) .. controls (383.14,225.36) and (383.14,225.36) .. (383.14,225.36) ;
\draw  [color={rgb, 255:red, 208; green, 2; blue, 27 }  ,draw opacity=1 ][line width=3.75] [line join = round][line cap = round] (275.14,69.36) .. controls (275.14,69.36) and (275.14,69.36) .. (275.14,69.36) ;
\draw  [color={rgb, 255:red, 208; green, 2; blue, 27 }  ,draw opacity=1 ][line width=3.75] [line join = round][line cap = round] (280.14,107.36) .. controls (280.14,107.36) and (280.14,107.36) .. (280.14,107.36) ;
\draw  [color={rgb, 255:red, 208; green, 2; blue, 27 }  ,draw opacity=1 ][line width=3.75] [line join = round][line cap = round] (285.14,143.36) .. controls (285.14,143.36) and (285.14,143.36) .. (285.14,143.36) ;
\draw  [color={rgb, 255:red, 208; green, 2; blue, 27 }  ,draw opacity=1 ][line width=3.75] [line join = round][line cap = round] (294.14,215.36) .. controls (294.14,215.36) and (294.14,215.36) .. (294.14,215.36) ;
\draw [color={rgb, 255:red, 208; green, 2; blue, 27 }  ,draw opacity=1 ] [dash pattern={on 4.5pt off 4.5pt}]  (206.14,132.36) -- (282.51,141.35) ;
\draw [shift={(284.5,141.58)}, rotate = 186.72] [color={rgb, 255:red, 208; green, 2; blue, 27 }  ,draw opacity=1 ][line width=0.75]    (10.93,-3.29) .. controls (6.95,-1.4) and (3.31,-0.3) .. (0,0) .. controls (3.31,0.3) and (6.95,1.4) .. (10.93,3.29)   ;

\draw (223,7.73) node [anchor=north west][inner sep=0.75pt]    {$\ell _{1}$};
\draw (185,82.73) node [anchor=north west][inner sep=0.75pt]    {$\ell _{2}$};
\draw (182,153.73) node [anchor=north west][inner sep=0.75pt]    {$\ell _{3}$};
\draw (233,237.73) node [anchor=north west][inner sep=0.75pt]    {$\ell _{4}$};
\draw (278,9.73) node [anchor=north west][inner sep=0.75pt]  [color={rgb, 255:red, 208; green, 2; blue, 27 }  ,opacity=1 ]  {$\ell $};
\draw (282,53.73) node [anchor=north west][inner sep=0.75pt]  [color={rgb, 255:red, 208; green, 2; blue, 27 }  ,opacity=1 ]  {$p\ =\ \lambda a+\mu b$};
\draw (438,170.4) node [anchor=north west][inner sep=0.75pt]  [color={rgb, 255:red, 74; green, 144; blue, 226 }  ,opacity=1 ]  {$V\ =\ p\lor \ell _{2}$};
\draw (233.67,28.4) node [anchor=north west][inner sep=0.75pt]    {$a$};
\draw (394.67,219.4) node [anchor=north west][inner sep=0.75pt]    {$b$};
\draw (102.67,120.4) node [anchor=north west][inner sep=0.75pt]  [color={rgb, 255:red, 208; green, 2; blue, 27 }  ,opacity=1 ]  {$q\ =\ V\ \land \ell _{3}$};
\draw (278,272.4) node [anchor=north west][inner sep=0.75pt]  [color={rgb, 255:red, 208; green, 2; blue, 27 }  ,opacity=1 ]  {$\ell \ =\ p\lor q$};

\end{tikzpicture}

%% file: brackets.bib
@book{sturmfels2008invariant,
  author    = {Sturmfels, Bernd},
  title     = {Algorithms in Invariant Theory},
  series    = {Texts \& Monographs in Symbolic Computation},
  edition   = {2},
  publisher = {Springer Vienna},
  date      = {2008-04-28},
  isbn      = {978-3-211-77416-8},
  doi       = {10.1007/978-3-211-77417-5},
  pages     = {VII, 197}
}

@incollection{stillman-exterior,
    AUTHOR = {Stillman, Michael},
     TITLE = {Computing in algebraic geometry and commutative algebra using
              {M}acaulay 2},
      NOTE = {International Symposium on Symbolic and Algebraic Computation
              (ISSAC'2002) (Lille)},
   JOURNAL = {J. Symbolic Comput.},
  FJOURNAL = {Journal of Symbolic Computation},
    VOLUME = {36},
      YEAR = {2003},
    NUMBER = {3-4},
     PAGES = {595--611},
      ISSN = {0747-7171,1095-855X},
   MRCLASS = {14Q15 (13P10)},
  MRNUMBER = {2004043},
MRREVIEWER = {Gerhard\ Pfister},
       DOI = {10.1016/S0747-7171(03)00096-8},
       URL = {https://doi.org/10.1016/S0747-7171(03)00096-8},
}

@Misc{M2,
          author = {Grayson, Daniel R. and Stillman, Michael E.},
          title = {Macaulay2, a software system for research in algebraic geometry},
          howpublished = {Available at \url{http://www2.macaulay2.com}}
 }

@article{ferraro2024invariantring,
  title={{The Invariantring package for Macaulay2}},
  author={Ferraro, Luigi and Galetto, Federico and Gandini, Francesca and Huang, Hang and Mastroeni, Matthew and Ni, Xianglong},
  journal={Journal of Software for Algebra and Geometry},
  volume={14},
  number={1},
  pages={5--11},
  year={2024},
  publisher={Mathematical Sciences Publishers}
}

@article{burr2024subalgebrabases,
  title={SubalgebraBases in Macaulay2},
  author={Burr, Michael and Clarke, Oliver and Duff, Timothy and Leaman, Jackson and Nichols, Nathan and Walker, Elise},
  journal={Journal of Software for Algebra and Geometry},
  volume={14},
  number={1},
  pages={97--109},
  year={2024},
  publisher={Mathematical Sciences Publishers}
}

@article{white1994grassmann,
  title={Grassmann—Cayley algebra and robotics},
  author={White, Neil L},
  journal={Journal of Intelligent and Robotic Systems},
  volume={11},
  pages={91--107},
  year={1994},
  publisher={Springer}
}

@inproceedings{thomas2023new,
  title={New Bracket Polynomials Associated with the General Gough-Stewart Parallel Robot Singularities},
  author={Thomas, Federico},
  booktitle={2023 IEEE International Conference on Robotics and Automation (ICRA)},
  pages={9728--9734},
  year={2023},
  organization={IEEE}
}

@article{agarwal2025flatland,
  title={A computer vision problem in flatland},
  author={Agarwal, Sameer and Connelly, Erin and Crannell, Annalisa and Duff, Timothy and Thomas, Rekha R},
  journal={arXiv e-prints},
  pages={arXiv--2501},
  year={2025}
}

@inproceedings{faugeras1995geometry,
  title={On the geometry and algebra of the point and line correspondences between $n$ images},
  booktitle={Proceedings of IEEE International Conference on Computer Vision},
author={Faugeras, Olivier and Mourrain, Bernard},
  pages={951--956},
  year={1995},
  organization={IEEE}
}
